\documentclass[12pt]{amsart}
\usepackage{amsfonts, amsmath, amssymb, amscd, latexsym, graphicx}

\newtheorem{thm}{Theorem}[section]
\newtheorem{prop}[thm]{Proposition}
\newtheorem{lem}[thm]{Lemma}

\newtheorem{defn}[thm]{Definition}

\renewcommand\>{\rangle}

\begin{document}

\title{Metric properties and distortion in wreath products}

\author{Jos\'e Burillo}

\address{Departament de Matem\`atica Aplicada IV, EETAC-UPC, C/Esteve Torrades 5,
08860 Castelldefels, Barcelona, Spain} \email{burillo@ma4.upc.edu}

\author{Eric L\'opez Plat\'on}
\address{Departament de Matem\`atica Aplicada IV, UPC, C/Jordi Girona 1-3,
08034 Barcelona, Spain} \email{eric.lopez@ma4.upc.edu}
\thanks{The authors are grateful for the support from MEC grant MTM2011--25955, and also from the UPC and the research group COMBGRAF}

\begin{abstract}
For a finitely generated regular wreath product, the metric is known, but its computation can be an NP-complete problem. Also, it is not known for the nonregular case.

In this article, a metric estimate is defined for regular wreath products which can be computed in polynomial time, based on the metrics of the factors. This estimate is then used to study the distortion of some natural subgroups of a wreath product. Finally, the metric estimate is generalized to the nonregular case.
\end{abstract}

\maketitle

\section*{Introduction}

Since the original works by Gromov and other authors in the 1980s, the study of infinite groups from the geometric point of view has experienced a great deal of development. The main tool to study groups as geometric objects is the Cayley graph of the group with a given set of generators, which can be given structure of metric space and whose metric properties are not yet completely well understood, and which give insights on the algebraic properties of the groups.

One of the concepts developed to study groups from the metric point of view is the concept of distortion of a subgroup in a group. This concept, analogous to the geometric concept of distortion of a submanifold, measures the difference between the two metric structures of a finitely generated subgroup inside a group. Namely, its own metric given by its generating set, which compares to the metric induced by the metric of the larger group. This gives rise to the concept of distortion function, which measures the difference between the two metrics. A subgroup is nondistorted if the two metrics are comparable (the distortion function is linear), a concept analogous to that of totally geodesic submanifolds of riemannian manifolds.

The concept of distortion appears already in Gromov's paper \cite{gromovasym}, and has been studied by several authors, such as Bridson \cite{bridsonfrac}, where it is shown that $n^r$ is the distortion function for a pair $G\subset H$ for any rational number $r>1$, or Sapir and Ol'shanskii (see \cite{olshdist} and \cite{sapir-olshdist}), where they give a description of which functions can be obtained as distortion functions of cyclic subgroups in finitely presented groups.

A metric estimate for a group is a function which is equivalent to the metric, up to multiplicative and additive constants. Such a function can be used to compute distortions and simplify other calculations.

A wreath product is the semidirect product of a group, and the group of maps from itself, or a set on which it acts, to another group. A wreath product is finitely generated if the two groups are finitely generated, and the action has a finite number of orbits. Computation of the exact value of the distance between two points in a wreath product involves solving a traveling salesman problem in the Cayley graph of the group. This is a well known and widely studied hard problem, as can be seen in \cite{TSPBook}, and hence the metric can be difficult to compute. The main result in this paper is that the distance given by the traveling salesman problem can be replaced by the weight of a minimum spanning tree, producing an equivalent value for the metric, but which can be computed in polynomial time, as can be seen in \cite{eisner-1997-mst}. This metric is then used to study the distortion of some natural subgroups of a wreath product. Studies of distortion for some other subgroups of particular wreath products 
have been done in \cite{olshwreath}. Finally, the metric estimate is generalized to the nonregular case.

\section{Background}

\subsection{Metric estimates}
Recall the definition of distortion function of a subgroup, due to Gromov in \cite{gromovasym}:
\begin{defn} Let $G$ be a finitely generated group, and $H<G$ a subgroup, also finitely generated. Define the \emph{distortion of $H$ in $G$} as:
$$
\Delta_H^G(n)=\max\left\lbrace \|x\|_H : x\in H,\ \|x\|_G\leq n \right\rbrace
$$
\end{defn}
As usual, the distortion function depends on the generating set. Two distortion functions for the same subgroup in a group, but with different generating sets (in either the group or the subgroup) differ only by multiplicative constants, and hence only the order of the distortion funcion is really defined. So we talk about quadratic, polynomial or exponential distortion.

The fact that the distortion is defined only up to multiplicative and additive constants, implies that it is not necessary to know the exact values of $\|x\|$, it is enough to compute them up to constants. This fact gives rise to the following definition:

\begin{defn}
 Let $G$ be a group and let $f,g:G\longrightarrow \mathbb{R}$ be two functions. We say that $f$ and $g$ are \emph{equivalent} (written $f\sim g$) if there exist two constants $C$ and $D$ such that
 $$
 \frac{g}{C}-D\leq f \leq Cg+D
 $$
\end{defn}

Observe that $\sim$ is an equivalence relation.

\begin{defn}
 Let $G$ be a finitely generated group. A map $f:G\rightarrow \mathbb{R}$ is called an \emph{estimate of the metric of $G$} or \emph{quasi-metric} if it is equivalent to $||\cdot||_G$, the word metric on $G$.
\end{defn}

With this definition, if we have quasi-metrics for both $G$ and $H$, we can redefine the distortion function:
$$
\Delta_H^G(n)=\max\left\lbrace E_H(x) : x\in H,\ E_G(x)\leq n \right\rbrace.
$$

\subsection{Wreath products}

The wreath product of two groups is a classical construction that appears in different problems, and has some interesting metric properties, depending on the metric properties of the two factors.

\begin{defn}
 Given two groups, $A$, $B$, and $\Omega$ a set on which $B$ acts, their \emph{restricted wreath product} $A\wr_\Omega B$ is defined to be the following semidirect product:
 \begin{equation*}
  \bigoplus_{x\in \Omega}A_x\rtimes B
 \end{equation*}

Where $B$ acts on the set of all maps from $\Omega$ to $A$ with finite support by changing the index.

\end{defn}

The unrestricted wreath product allows infinite support, but for the purpose of this paper we will only use the restricted version. Also, $B$ acts naturally on itself by left multiplication, and we will note simply by $A\wr B$ the special case where $\Omega=B$, which is called \emph{regular}.

An alternative way of understanding wreath products arises from a classical example. The lamplighter group $\mathbb{Z}_2\wr\mathbb{Z}$ can be understood as a lamplighter on an infinite street (the Cayley graph of $\mathbb{Z}$), with a light on every integer, which is either lit or dark. A word in the group is a set of instructions for the lamplighter, which can be either move to the next or previous light (a generator of $\mathbb{Z}$) or turn on or off the light you are currently at (the generator of $\mathbb{Z}_2$).

This would be the identity element:

 \begin{center}
 \includegraphics[scale=1,keepaspectratio=true]{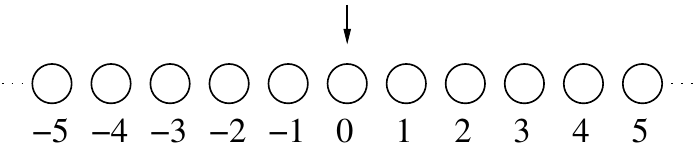}
\end{center}

Where the blank circles represent the $0$ in $\mathbb{Z}_2$ and the arrow represents the lamplighter's final position.
After a word like $t^2atat^{-4}at^{-2}at$ the element would look like this:

 \begin{center}
 \includegraphics[scale=1,keepaspectratio=true]{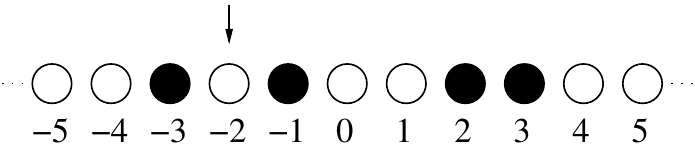}
\end{center}

Where the filled circles represent the element $1$ of $\mathbb{Z}_2$.

This viewpoint can be generalized, and in general a restricted product $A\wr B$ be understood as a lamplighter (or a pointer) traveling through the Cayley graph of $B$ and adding elements of $A$ when prompted. Notice that this would only generate elements with finite support, so it can only be applied to the restricted wreath product.

Such point of view can still be used when the product is not regular, considering the lamplighter traveling through $\Omega$. If the action is not transitive, however, the wreath product will need a different lamplighter for each orbit.

\begin{prop}[Presentation]

If $A=\langle S_A | R_A \rangle$ and $B=\langle S_B | R_B \rangle$, and $B$ acts transitively on $\Omega$, choose $x\in\Omega$ a starting point, then the wreath product $A\wr_\Omega B$ admits the following presentation:
\begin{equation}
 A\wr_\Omega B = \langle S_A, S_B | R_A, R_B, [a_1^{b_1}, a_2^{b_2}] (\text{if } xb_1\neq xb_2)  \rangle
\end{equation}

Where $[g,h]$ is the commutator $ghg^{-1}h^{-1}$, and $g^h$ is the conjugation $h^{-1}gh$.
\end{prop}

The commutators can be reduced to just use the generators of $A$, but not for $B$. Because of this, it is known that $A\wr B$ is finitely presented if and only if either $B$ is finite or $A$ is trivial. In the non-regular case the conditions are a bit more complex.

This presentation has an starting point in $\Omega$, $x$, which is understood to be $1_B$ for the regular case. If the action had more than one orbit, then we would need a starting point in each orbit, and a different set of generators $S_A$ assigned to each starting point. All starting points move at the same time (because $B$ acts on the whole set), but the generators of $A$ have a mark stating on which orbit they are writing.

\begin{prop}[Presentation]

If $A=\langle S_A | R_A \rangle$ and $B=\langle S_B | R_B \rangle$, and $B$ acts on $\Omega$, choose $x_i\in\Omega$ a set of starting points, one for each orbit, then the wreath product $A\wr_\Omega B$ admits the following presentation:
\begin{equation*}
 A\wr_\Omega B = \langle {S_A}_{x_i}\text{ for } i\in \Omega/B, S_B | {R_A}_{x_i}\text{ for } i\in \Omega/B, R_B, [a_1^{b_1}, a_2^{b_2}] (\text{if } xb_1\neq xb_2)  \rangle
\end{equation*}

\end{prop}

\section{Metric in regular restricted wreath products}

We are interested in studying the word metric in this kind of products. First of all, for the word metric to make sense, we need the group to be finitely generated. Just looking at the presentation it is clear that we need both $A$ and $B$ to be finitely generated. We also need the product to be restricted, and the action of $B$ on $\Omega$ to have finitely many orbits, but in this section we will restrict ourselves to the regular and restricted case.

In this case, the metric is already known. From an algebraic point of view, take an element $x$ of $A\wr B$. This element is a map from $B$ to $A$ with finite support, that is, a choice of finitely many elements $b_1, \dots, b_s$ of elements of $B$ mapped to $a_1,\dots, a_s$, and possibly a final element $b_f$ of $B$. We have an obvious normal form:
\begin{equation*}
 x=\left(\prod_{i=1}^sb_ia_ib_i^{-1}\right) b_f
\end{equation*}

which is uniquely defined, up to permutation, given the $a_i$, $b_i$ and $b_f$. Consider this normal form from the lamplighter point of view for the wreath product. What it does is go to every $b_i$, apply an $a_i$ there and go back to the origin. There is no reason to visit the same $a_i$ twice, so this makes sense. However, this word is not geodesic. In some cases, it would be easier to find a path starting at the identity, visiting each $b_i$ (in any order) and ending in $b_f$, and being shorter that going back to the origin after each step.

This can be repeated as long as the path we have is not optimal. So the actual metric will appear when we have the shortest path from the identity to $b_f$ visiting all the $b_i$ in any order. Finding such a path is a classical problem in algorithmics, known as the Traveling Salesman Problem. Since the actual Traveling Salesman Problem looks for a cycle (ending at the starting point), and we want it to have particular starting and ending points, we will consider the Traveling Salesman Path Problem. We will denote by $\tau(e,b_i,b_f)$ the length of the path which provides a solution for the Problem for this particular case, i.e., the length of a path in the Cayley graph of $B$ which starts at $e$, visits each one of the vertices $b_i$ and ends at $b_f$.

With this in mind, and considering that when we visit one of the $b_i$ we can build $a_i$ in just one visit and in an optimal way, the metric in $A\wr B$ can be calculated.

\begin{prop}
 Let $x\in A\wr B$ be an element of the form $\left(\prod_{i=1}^sb_ia_ib_i^{-1}\right) b_f$. Then

 \begin{equation*}
  \lVert x \rVert_{A\wr B}=\sum_{i=1}^s\lVert a_i\rVert_A + \tau(e, b_i, b_f)
 \end{equation*}

 is the word metric for $x$.

\end{prop}

A proof of this can be seen in \cite{olshwreath}. Although it is the exact metric, this number presents a few complications. For instance, it requires solving an $NP$-hard problem, the Traveling Salesman Path Problem, which will be the main problem. It also requires knowing the exact distance in $A$ and $B$, and finally, it is hard to extend to the nonregular case. These problems can be solved using metric estimates, and although the metric will not be exact, but again an estimate, it will still be useful for many situations.

The basic idea is finding something which is equivalent (up to constants) to the solution of the Traveling Salesman Path Problem, and which can be calculated in polynomial time. There are several approximation algorithms which run in polynomial time, but for simplicity reasons, since we are not bounded to get the smallest possible constants, we will use the minimum spanning tree:

\begin{defn}
 Given a weighted graph $\Gamma$, a \emph{minimum spanning tree} $T$ is a connected graph containing all the vertices in $\Gamma$ such that its total weight is minimum on all graphs satisfying this condition.
\end{defn}

\begin{lem}
 Let $X$ be a metric space, and $x_1, \dots, x_n$ be $n$ points in $X$. Build the complete graph $K$ on $n$ vertices, with weights $w(i,j)=d_X(x_i, x_j)$. Let $\mu(x_1,\dots, x_n)$ be the weight of a minimum spanning tree in $K$.

 Then there is a constant $C$ such that
 \begin{equation*}
  \frac{\mu(x_1,\dots,x_n)}{C} \leq \tau(x_1,x_2, \dots, x_{n-1}, x_n) \leq C\,\mu(x_1,\dots,x_n)
 \end{equation*}

\end{lem}

\begin{proof}
 Let $T$ be a minimum spanning tree in $K$, and $P$ a minimum path in $X$, whose total length is $\tau(x_1,x_2, \dots, x_{n-1}, x_n)$. Since $P$ is an optimal path, between any two consecutive vertices it must take a minimum path, so $P$ can be mapped into $K$, let's call this path $P'$.

 In $K$, since $P'$ is a path, it is probably a tree. If it is not, because it is has some cycles, it can be made into a tree by removing edges, while keeping it connected. But the resulting tree is an spanning tree, so $T$ must have less total weight. Hence $\mu(x_1, \dots, x_n)\leq \tau(x_1,x_2, \dots, x_{n-1}, x_n)$.

%
%
For the opposite inequality, travel $T$ along the edges, from $x_1$ to $x_n$ visiting every vertex in between. The resulting path has at most twice the weight of $T$. This path can be mapped into $X$, with each edge being an optimal path between its two endpoints. The image will be a path visiting all $x_i$, beginning at $x_1$ and ending at $x_n$, so it must be longer than $P$. Therefore $\tau(x_1,x_2, \dots, x_{n-1}, x_n)\leq 2\mu(x_1, \dots, x_n)$

\end{proof}

Using this lemma, we can define an estimate for the metric of $A\wr B$. Actually, we can define several estimates, based either on the exact metric or on estimates for either $A$ or $B$.

\begin{thm}
 Let $A$ and $B$ be finitely generated groups, and be $E_A$ and $E_B$ estimates of their respective metrics. Let $\mu(e,b_i,b_f)$ be the weight of the minimum spanning tree as shown in the previous lemma, and let $\hat{\mu}(e,b_i,b_f)$ be the weight of a similarly calculated spanning tree using $E_B$ in place of $\lVert.\rVert_B$. Given $x=(\prod_{i=1}^sb_ia_ib_i^{-1}) b_f$, all the following are metric estimates for $x$.
 \begin{enumerate}
  \item $\sum_{i=1}^s\lVert a_i\rVert_A + \mu(e, b_i, b_f)$
  \item $\sum_{i=1}^sE_A(a_i) + \tau(e, b_i, b_f)$
  \item $\sum_{i=1}^sE_A(a_i) + \mu(e, b_i, b_f)$
  \item $\sum_{i=1}^s\lVert a_i\rVert_A + \hat{\mu}(e, b_i, b_f)$
  \item $\sum_{i=1}^sE_A(a_i) + \hat{\mu}(e, b_i, b_f)$
  \item $\sum_{i=1}^s\lVert a_i\rVert_A + \mu(e, b_i, b_f) + E_B(b_f)$
  \item $\sum_{i=1}^sE_A(a_i) + \hat{\mu}(e, b_i, b_f) + E_B(b_f)$
 \end{enumerate}

\end{thm}

Observe that $\sum_{i=1}^s\lVert a_i\rVert_A + \tau(e, b_i, b_f)$ is the exact distance from the identity to the element $x$. These estimates have been obtained by replacing $\lVert a_i\rVert_A$ by $E_A(a_i)$ and by replacing $\tau$ by $\mu$ or by $\hat\mu$. The last two estimates, with the addition of $E_B(b_f)$, are also equivalent, and will be useful for the generalization to nonregular wreath products.

\begin{proof}
 First, observe this property of the equivalence of functions: if $f\sim f'$ then $f+g\sim f'+g$.

  The first thing is to check that $\mu\sim \widehat{\mu}$. For this, consider $K$ the complete graph on the support, with weights the distances in $B$, and $T$ a minimum spanning tree on $K$. Also, repeat the same process using $E_B$ instead of the actual word metric. That will give rise to another complete graph $K'$ with weights the estimates of the distance, on which you can compute a new minimum spanning tree, $T'$. Observe that $\mu(e, b_i, b_f)$ is the total weight of $T$, while $\hat{\mu}(e, b_i, b_f)$ is the total weight of $T'$.

 Taking $T$, the tree can be mapped (edge-wise) to $K'$. Even though the weights will be different, it is just up to a constant $C$. The total weight of the image will be greater than the weight of $T'$, because $T'$ is minimal. So $\mu(e, b_i, b_f)>C\cdot\hat{\mu}(e, b_i, b_f)$  The same works for the opposite inequality.

 For the last two estimates, observe that $E_B(b_f)<\hat{\mu}(e, b_i, b_f)$, because the $\hat{\mu}$ contains a path from $e$ to $b_f$.

 Finally, to check that we can use the estimate for $A$ instead of the actual metric, when changing $\sum_{i=1}^s\lVert a_i\rVert_A$ for $\sum_{i=1}^sE_A(a_i)$, the multiplicative constant does not change. From the additive constant $D$, we will have a term $sD$, with $s$ not being a constant. However, since $s$ is the size of the support, it will always be less than the size of the minimum spanning tree.

\end{proof}

\section{Applications: subgroup distortion}

\label{sectiondistortion}
Using this metric estimate, we can compute distortions for the subgroups in $A\wr B$. Doing so in general is no easy task, because the subgroups of a wreath product might depend on the product itself, and they are not necessarily a product of subgroups of $A$ and $B$. However, we can study some particular, natural cases.

\subsection{Cyclic subgroups}

Given $A\wr B$ finitely presented groups, and $x\in A\wr B$, we want to compute the distortion function for $\<x\>$, the subgroup generated by $x$, in the ambient group.

\begin{thm}
 Let $A$ and $B$ be finitely presented groups, and $x\in A\wr B$ an element from their wreath product, and assuming $x=\left(\prod_{i=1}^sb_ia_ib_i^{-1}\right) b_f$, we have that $\Delta_{\<x\>}^{A\wr B}$ equals

 \begin{itemize}
  \item $\max_{a_i}(\Delta_{\<a_i\>}^A(n))$, when $b_f$ is either the identity or a torsion element in $B$.
  \item $\Delta_{<b_f>}^B$, if $b_f$ is not a torsion element, and the size of the support of $x^n$ stabilizes after some $n$.
  \item $n$ (non-distorted), otherwise.
 \end{itemize}

\end{thm}

\begin{proof}

Considering $x=\left(\prod_{i=1}^sb_ia_ib_i^{-1}\right) b_f$, we have to study the behavior of the powers $x^n$ of $x$.

Since $b_f$ will be marking the starting point for the next copy of $x$, the first consideration is whether $b_f$ is back at the identity or not. If $b_f$ is the identity, then $x=\left(\prod_{i=1}^sb_ia_ib_i^{-1}\right)$, and $x^n=\left(\prod_{i=1}^sb_ia_i^nb_i^{-1}\right)$, because each iteration of $x$ will be starting from the same point. In this case, it is clear that the elements of $B$ do not change, the only powers appearing are powers of the $a_i$. Then

\begin{equation*}
 \Delta_{\<x\>}^{A\wr B}(n)=\max_{a_i}\Delta_{\<a_i\>}^A(n)
\end{equation*}

Also observe that if $b_f$ is a torsion element in $B$, and $b_f^m=1$, then $x^m$ ends on the identity, and we can instead compute the distortion for $\<x^m\>$, which will be equivalent.

If $b_f$ is not a torsion element, then each iteration of $x$ starts from a different point. Even though, for a particular element of the support, two or more instances of $x$ might overlap, only a finite amount will. So only bounded powers of the $a_i$ may appear. This overlapping can imply, however, than the elements in the support cancel out. Some of them, at the beginning or the end, might not cancel, while the elements in the middle will. Formally, if $b_i$ is an element in the support of $x$, call $b_{i_k}$ the $b_i$ as written from the beginning of the $k-th$ iteration. That is, $b_{i_k}=b_f^{k-1}b_i$. Then, suppose there exist $m_1, m_2>0$ and an $n_0$ such that for $n>n_0$, for $m_1\leq k \leq n-m_2$ we have $a_{i_k}=\prod_l a_{i_l}=1$, for all $i$. It is easy to see that this happens if and only if the size of the support stabilizes after $n_0$.

In this case, as $n$ grows past $n_0$, the $a_i$ are fixed, while the $b_i$ and $b_f$ change. Consider the metric estimate:

\begin{equation*}
 E(x^n)=\sum_{i=1}^s \lVert a_i \rVert + \hat{\mu}(e, b_i, b_f^n)
\end{equation*}

The sum of $a_i$ does not change, while the minimum spanning tree has an edge proportional to $b_f^n$ (see figure), while the other edges stay constant (in size). As such, the $\hat{\mu}(e, b_i, b_f^n)$ (and the whole estimate) grows as $b_f^n$. But notice that $b_f^n$ might be distorted, so the distortion $\Delta_{\<x\>}^{A\wr B}$ will be the same as $\Delta_{\<b_f\>}^{B}$

For a graphic representation, if this was $x$

\begin{center}
 \includegraphics[scale=1,keepaspectratio=true]{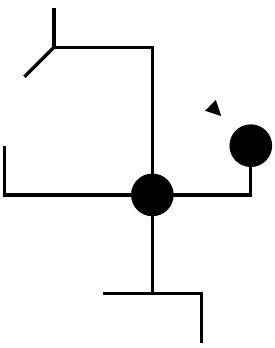}
\end{center}

where the black circles are $e_B$ and $b_f$, then this could be $x^n$, supposing all elements from the support in the intermediate steps were eliminated along the way.

\begin{center}
 \includegraphics[scale=1,keepaspectratio=true]{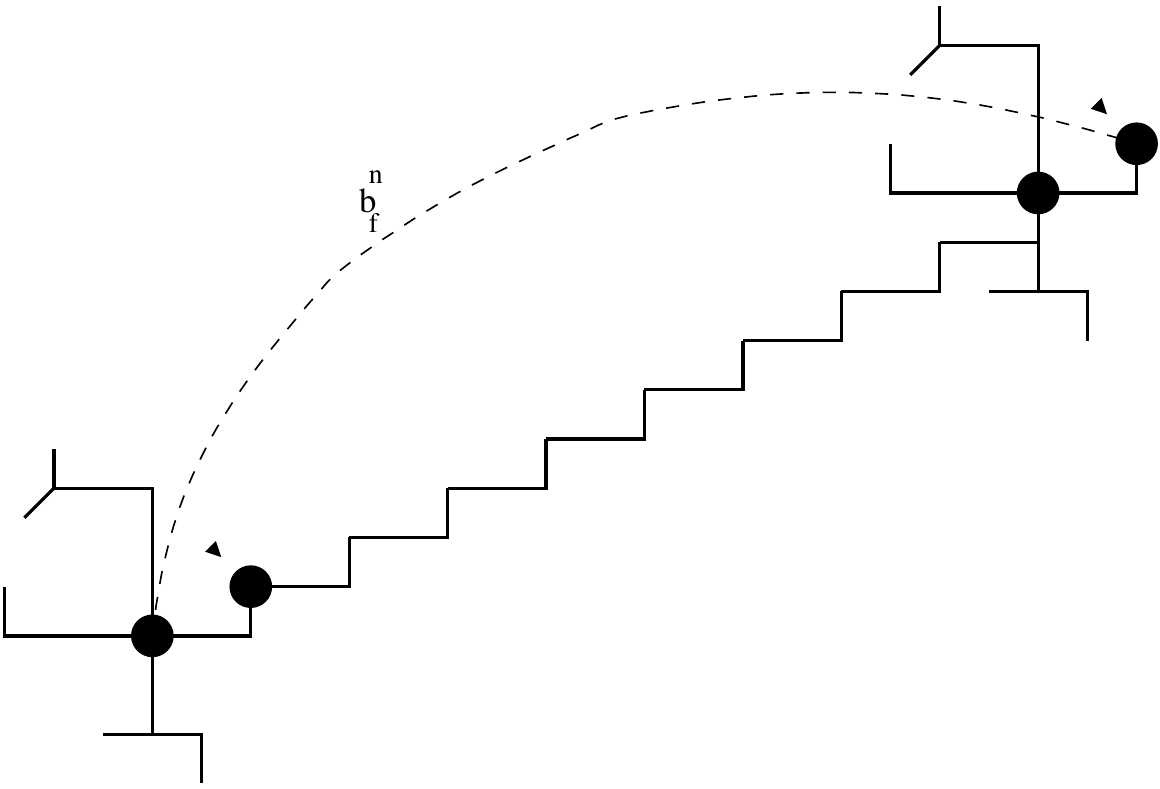}
\end{center}

In the case where the support does not stabilize, at each iteration of $x$ in $x_n$ there has to be an $a_{i_k}$ which is not the identity in any future iteration. Notice that if this is true for any $a_{i_k}$, it will be also true for $a_{i_{k+1}}$, except when close enough to $n$. As in the previous case, only finite powers of the $a_i$ appear, but in this case the $\hat{\mu}(e, b_i, b_f^n)$ must visit each iteration of $b_f^k$ individually (see figure), so it will be proportional to $n\lVert b_f\rVert$ instead of $b_f^n$. In this case, the total metric estimate $E(x^n)$ will be proportional to $n$ times $E(x)$. In this case, the subgroup is undistorted.

For this case, with an $x$ similar to the previous example, but with an element in the support which is never eliminated (the white dot)

\begin{center}
 \includegraphics[scale=1,keepaspectratio=true]{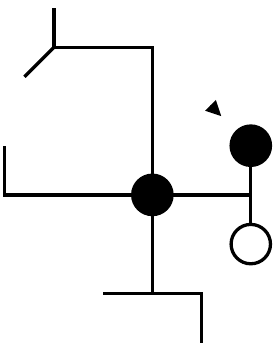}
\end{center}

In this case, the tree must visit each iteration of $b_f$, therefore even if $b_f^n$ can be written with a shorter word, the length of his element will still be proportional to $n\lVert b_f\rVert$.

\end{proof}

\subsection{Subgroups of the factors}

Another natural and interesting type of subgroups of a wreath product are those arising from the subgroups of one of the factors and taking again a wreath product.

That is, let $A$ and $B$ be groups, and $A\wr B$ their wreath product. Now consider subgroups $A'<A$ and $B'<B$. We can now study the distortion of subgroups of the form $A'\wr B$, $A\wr B'$ or even $A'\wr B'$ inside $A \wr B$.

As one would expect, these distortion functions will depend greatly on the distortion functions for $A'<A$ and $B'<B$. They will also depend on whether they are superadditive.

\begin{defn}

A function $f:\mathbb{N} \rightarrow \mathbb{N}$ is said to be \emph{superadditive} if it satisfies $f(x+y)\geq f(x)+f(y)$ for every $x,y\in \mathbb{N}$.

\end{defn}

Since we are talking of distortion functions, which we can change up to constants by changing the presentation, we will only be interested in whether the functions are equivalent to a superadditive function.

\begin{prop}

 Let $A$ and $B$ be finitely generated groups, and $A'<A$ a finitely generated subgroup. Let $\Delta_{A'}^A(n) \sim f(n)$. Then $f(n)\leq \Delta_{A'\wr B}(n)\leq nf(n)$, and in particular $\Delta_{A'\wr B}(n) \sim f(n)$ if $f(n)$ is superadditive.
\end{prop}

\begin{proof}
 It is clear than $\Delta_{A'\wr B}(n)\geq f(n)$, because we have $A'$ inside the wreath product, for example as elements of $A$ in the identity index, and trivial elements everywhere else, and using the metric estimate, this elements will have the metrics of $A'$ and $A$. Because of that, the element that gave the worst case metric for $A'<A$ will still give the same distortion. However, the distortion could be greater.

 To check this, we will use the metric estimate. Let $x=\left(\prod_{i=1}^sb_ia_ib_i^{-1}\right) b_f$, with $a_i\in A'$, be an element in $A'\wr B$ with $E_{A'\wr B}(x)<n$. Then:

$$
  E_{A'\wr B}(x)=\sum_{i=1}^s E_{A'}(a_i)+\mu(e, b_i, b_f)
  \leq \sum_{i=1}^s f(E_{A}(a_i))+\mu(e, b_i, b_f)
$$

Observe that by definition of distortion, $E_{A'}(a)<f(E_A(a))$. Also, since $f$ is an increasing function, $f(n)>n$ for any $n$. Now, if $f$ was superadditive, or equivalent to a superadditive function, this calculation could be continued:

\begin{equation*}
 \sum_{i=1}^s f(E_{A}(a_i))+\mu(e, b_i, b_f)\leq f\left(\sum_{i=1}^s E_{A}(a_i)\right)+\mu(e, b_i, b_f)<f(E_{A\wr B}(x))
\end{equation*}

If $f$ was not equivalent to a superadditive function, the continuation could be instead:

\begin{equation*}
 \sum_{i=1}^s f(E_{A}(a_i))+\mu(e, b_i, b_f)\leq sf(\sum_{i=1}^s E_{A}(a_i))+\mu(e, b_i, b_f)<sf(E_{A\wr B}(x))
\end{equation*}

 and using that $s$ is less than the length of $x$, the theorem is proven.

\end{proof}

An analogous theorem can be proven for $B'<B$, although one must take care of the possibility of the $\mu$ being different when calculated on $B$ and $B'$.

\begin{prop}

 Let $A$ and $B$ be finitely generated groups, and $B'<B$ a finitely generated subgroup. Let $\Delta_{B'}^B(n) \sim f(n)$. Then $f(n)\leq\Delta_{A\wr B'}^{A\wr B}(n)\leq nf(n)$, and in particular $\Delta_{A\wr B'}^{A\wr B}(n) \sim f(n)$ if $f(n)$ is superadditive.
\end{prop}

\begin{proof}
 The proof is very similar to the previous theorem, and we will only show that $\mu'<f(\mu)$, with $\mu'$ being the one calculated on $B'$.

 Let $x=\left(\prod_{i=1}^sb_ia_ib_i^{-1}\right) b_f$, with all the $b_i$ and $b_f$ in $B'$. We know that $\mu'(e, b_i, b_f)$ is the sum of some distances (at most $s+1$ distances) in $B'$. However, the $\mu(e, b_i, b_f)$, calculated in $B$, can contain different edges, not just different distances for the same edges.

 Because of this, consider $\sum_{i=1}^{s+1}d'_{B'}$ the distances of the $\mu'(e, b_i, b_f)$, and $\sum_{i=1}^{s+1}d_{B}$ the distances of the $\mu(e, b_i, b_f)$ in $B$. Furthermore, consider the sum of distances $\sum_{i=1}^{s+1}d_{B'}$, which are the distances corresponding to the edges in the $\mu(e, b_i, b_f)$, but with distances on $B'$.

 Now, since $\sum_{i=1}^{s+1}d_{B'}$ are not from a $\mu$, we have:

 \begin{equation}
  \mu'(e, b_i, b_f)=\sum_{i=1}^{s+1}d'_{B'}<\sum_{i=1}^{s+1}d_{B'}<\sum_{i=1}^{s+1}f(d_{B})
 \end{equation}

 Which will be at most $sf(\mu(e, b_i, b_f))$, and $f(\mu(e, b_i, b_f))$ if $f$ is superadditive.

\end{proof}

Finally, we can mix the two propositions in a more general result.

\begin{thm}
 Let $A$ and $B$ be finitely generated groups, and $A'<A$ and $B'<B$ finitely generated subgroups with respective distortions $\Delta_{A'}^A(n)=f(n)$ and $\Delta_{B'}^B(n)=g(n)$. Then we have $\Delta_{A'\wr B'}^{A\wr B}(n)\geq (f+g)(n)\sim max(f(n),g(n))$.

 Also, depending on whether $f$ and $g$ are equivalent to a superadditive function, we have:

\begin{itemize}
 \item if both $f$ and $g$ are superadditive, $\Delta_{A'\wr B'}^{A\wr B}(n)\leq max(f(n),g(n))$
 \item if $f$ is superadditive, but $g$ is not, $\Delta_{A'\wr B'}^{A\wr B}(n)\leq max(f(n),ng(n))$
 \item if $g$ is superadditive, but $f$ is not, $\Delta_{A'\wr B'}^{A\wr B}(n)\leq max(nf(n),g(n))$
 \item if neither are superadditive, $\Delta_{A'\wr B'}^{A\wr B}(n)\leq max(nf(n),ng(n))$
\end{itemize}

\end{thm}

\begin{proof}
 First, it is clear than the distortion is at least the worst of $A'<A$ and $B'<B$, because both $A'$ and $B'$ are inside the wreath product.

 For the opposite inequality, take $x=\left(\prod_{i=1}^sb_ia_ib_i^{-1}\right) b_f\in A'\wr B'$.
 Using the arguments from the two previous propositions, we can get:

 \begin{equation*}
  E_{A'\wr B'}(x)=\sum_{i=1}^sE_{A'}(a_i) + \mu'(e, b_i, b_f)<s^{e_f}f(\sum_{i=1}^sE_{A}(a_i)) + s^{e_g}g(\mu'(e, b_i, b_f))
 \end{equation*}

 with the $e_f$ and $e_g$ being $0,1$ depending or whether $f$ and $g$ are superadditive. From here, we can use the fact that $f(n), g(n)\geq n$ to draw the conclusion.
\end{proof}

However, this results use superadditivity as a condition, but it could be that the same bound worked even without superadditivity, or that there are no group inclusions with distortion function non-equivalent to a superadditive function. These are both false, and a proof of the second one can be seen in \cite{superadditive}, and we will see an example where a non-superadditive distortion in one factor gives place to a greater distortion for the wreath product.

\begin{prop}
 There are finitely generated groups $A,B$ and a subgroup $A'<A$ such that $f(n)=\Delta_{A'}^A(n)$ is not superadditive and $\Delta_{A'\wr B}^{A\wr B}(n)>f(n)$.
\end{prop}

An example of an inclusion with distortion function not equivalent to any superadditive function is needed. In \cite{superadditive} it is proven constructively that such an inclusion exists. We will use that example as $A'<A$. Let $g(n):\mathbb{N}\rightarrow\mathbb{N}$ be a function greater than any recursive function, and let $f(n):\mathbb{N}\rightarrow\mathbb{N}$ be a function that grows as fast as $g(n)$, but which is linear on an infinite sequence. Consider the sequence $d_i\in\mathbb{N}$, such that $f(d_i)=g(d_i)$ but $f(d_i-1)<2d_i-1$. Again, Davis and Olshanskii show an explicit example satisfying these properties, with an inclusion that has $f$ as distortion function.

First we will need a technical lemma:

\begin{lem}
 In these hypotheses, given constants $C,D>0$, and defining $k_i$ such that $k_id_i+k_i<d_{i+1}-1\leq (k_i+1)d_i+(k_+1)$, there exists an $i\in \mathbb{N}$ such that $kf(d_i)>Cf(d_{i+1}-1)+D$.
\end{lem}

\begin{proof}
 Assume it is false. That is:

 \begin{equation*}
  \forall i\in\mathbb{N}\quad k_if(d_i)\leq Cf(d_{i+1}-1)+D \leq C(2d_{i+1}-1)+D \leq C(2(k_i+1)d_i+(k+1))+1)+D
 \end{equation*}
 hence
 \begin{equation*}
  f(d_i)\leq C\left(2\frac{k+1}{k}d_i+\frac{2k+3}{k}\right)+D \leq C(4d_i+5)+D
 \end{equation*}

 Which would make $f(d_i)$ linear, contradicting the fact that $f(d_i)=g(d_i)$, which is greater than any recursive function.

 \end{proof}

 And using this we can prove the result.

 \begin{proof}[Proof of the Proposition]
 Let $A'<A$ be the subgroup inclusion with $\Delta_{A'}^A(n)=f(n)$, where $f(d_i)=g(d_i)$, which is greater than any recursive function, but $f(d_i-1)\leq2(d_i-1)+1$. We will use $E(x)=\sum_{i=1}^s\lVert a_i\rVert_A + \mu(e, b_i, b_f)$ as metric estimate for an element of the form $x=\left(\prod_{i=1}^sb_ia_ib_i^{-1}\right) b_f$.

 Take $B=\mathbb{Z}$, we will prove that for any $C,D>0$ there exists an $x\in A'\wr \mathbb{Z}$ such that $E_{A'\wr \mathbb{Z}}(x)>Cf(E_{A\wr \mathbb{Z}}(x))+D$, which will show that the distortion of $A'\wr \mathbb{Z}$ in $A\wr\mathbb{Z}$ is strictly greater than $f$.

 Given $C,D>0$, take the $i$ satisfying the lemma. Take $h\in A'$ such that $\lVert h\rVert_A\leq d_i$ and $\lVert h\rVert_{A'}= f(d_i)$, that is, an element achieving the distortion. Take $k_i$ from the lemma, and let $x=(ht)^k$. Now, using the lemma:

 \begin{align*}
   E_{A\wr\mathbb{Z}}(x)&=\sum_{j=1}^{k_i}\lVert h\rVert_{A'}+\mu(e,k)=k_if(d_i)+k>k_if(d_i)>Cf(d_{i+1}-1)+D \\
    &\geq  Cf(k_id_i+k_i)+D=Cf\left(\sum_{j=1}^{k_i}\lVert h\rVert_{A}+\mu(e,k)\right)+D=Cf(E_{A\wr \mathbb{Z}}(x))+D
 \end{align*}

 \end{proof}

\section{On the non-regular case}

Finally, the metric estimation defined for wreath products can be extended for non-regular wreath products. That is, where the action of the group $B$ on $\Omega$ is not regular. In this case, the actual metric is not just computationally hard, but it is also difficult from a theoretic point of view. However, the metric estimate can be generalized with a bit of work so that it still works and the results on distortion can be analogously proven with it.

Observe, first, that in the non-transtive case, there are several orbits. In this case, the presentation needs a copy of the generators of $A$ for each orbit. For this reason, even if the action is not transitive, it must have a finite number of orbits. In the lamplighter interpretation, there is a pointer in each orbit (an starting point $e_j$ must be chosen), and when an element of $B$ acts on $\Omega$ it moves each pointer accordingly. Naming the orbits $\Omega_1, \dots \Omega_{\omega}$, we will name this sets of generators as $A_j$, and a word $a_j$ will be interpreted as being a word from $A$ on the generators $A_j$.

With this, and with $s$ being the support, we can write any element as:

\begin{equation*}
 x=\left(\prod_{i=1}^sb_ia_{j_i}b_i^{-1}\right) b_f
\end{equation*}

with $i_j$ being the orbit corresponding for the $i$th element of the support.

Although any element can be written as such, it is unclear how to deduce an optimal word from this form, and it is not as simple as reproducing the Traveling Salesman Problem, since we don't have an obvious metric space, and there might be several words in $B$ that work for the elements in a given orbit, but might not be optimal for the other orbits.

To build an analogous estimation in this case, we first need a graph equivalent to the Cayley graph of $B$. Given $\Omega$, consider the graph $(\Omega, E)$ with vertices in the elements of $\Omega$ and a directed edge between two vertices if there is a generator of $B$ sending one element to the other. In the transitive case, this will be the Schreier graph of the stabilizer of a point. In the non-transitive case, the graph will have several connected components, as many as orbits, and each one will be equivalent to the Schreier graph of the stabilizer of a point from that orbit.

Given an element $x$ in the wreath product, let $\mu_j(e_j, e_jb_{i_j}, e_jb_f)$ be the minimum spanning tree that spans $e_j$, the starting point in orbit $j$, the vertices in the orbit that are in the support, that is $e_jb_i$, and the ending point in that orbit, which is $e_jb_f$. This minimum spanning tree is calculated in the same form that in the regular case. That is, generate a complete graph with the same vertices, and their distances as edges, and calculate a minimum spanning tree on that graph.

\begin{thm}
Let $A, B$ be finitely generated groups, and $\Omega$ be a $B$-set with a finite number of orbits $1,\dots, \omega$. Consider the wreath product $A\wr_{\Omega}B$, and an element $x=\left(\prod_{i=1}^sb_ia_{j_i}b_i^{-1}\right) b_f$ of the product.

Then

\begin{equation*}
 E_{A\wr_{\Omega}B}(x)=\sum_{i=i}^s\lVert a_i\rVert_A + \sum_{j=i}^{\omega}\mu_j(e_j, e_jb_{i_j}, e_jb_f) + \lVert b_f\rVert_B
\end{equation*}

is an estimate of the metric.

\end{thm}

\begin{proof}
 For one bound, check that $E_{A\wr_{\Omega}B}(x)$ is greater than the actual metric by constructing the following word: for each orbit, follow the $\mu_j$ in any order, and when arriving to an element in the support add an optimal word in $A$ for the corresponding element. Finally, add an optimal word for $b_f\in B$ The result is a word for $x$ with length $E_{A\wr_{\Omega}B}(x)$, so the actual length must be smaller.

 For the opposite bound, we will check that multiplying a given element $x$ by a generator $a_j$ or $b$ will not make $E_{A\wr_{\Omega}B}(x)$ grow more than by $\omega+1$, which is a constant depending only on the group. With that, writing the optimal word for $x$ can not make $E(x)$ greater than $(\omega+1)\lVert x \rVert$.

 Given an element $x$, multiplying it by $a_j$ will increase or decrease by one the $\lVert a_i\rVert$ corresponding to $b_f$. The sum of $\mu$, however, will not be altered, because $b_f$, the current position, was already counted in all the trees, and adding or removing it from the support will not change this.

 If the element is a generator $b\in B$, each $\mu$ can change at most by 1, since $b_f$ is changing by one generator, and the same happens to $\lVert b_f \rVert_B$. Since there are $\omega$ instances of $\mu$, the total value can only change by at most $\omega +1$, which is a fixed constant given the group.

\end{proof}

Also, using this metric estimate, the results from Section \ref{sectiondistortion} can be extended to the nonregular wreath product.

\bibliography{ericsrefs}
\bibliographystyle{plain}

\end{document}